\documentclass[a4paper,10pt,twoside]{article}

\usepackage[T1]{fontenc}
\usepackage[utf8]{inputenc}
\usepackage{amsfonts}
\usepackage{amssymb}
\usepackage{amsmath}
\usepackage{amsthm}
\usepackage{graphicx}

\usepackage{latexsym}

\usepackage{verbatim}

\makeatletter
\renewcommand{\@evenfoot}{\hfil - \thepage\ - \hfil}
\renewcommand{\@oddfoot}{\hfil - \thepage\ - \hfil}
\makeatother

\newcounter{theo}

\newtheorem{thm}{Theorem}
\newtheorem{coroll}{Corollary}
\newtheorem{lemme}{Lemma}

\usepackage{comment}

\newcommand{\bP}{\mathbb{P}}

\newcommand{\indic}{{\bf 1}} % fonction indicatrice

\newcommand{\limites}[2]{\overset{#1}{\underset{#2}{\longrightarrow}}}

\newcommand{\R}{\mathbb{R}}
\newcommand{\N}{\mathbb{N}}
\newcommand{\Z}{\mathbb{Z}}

\newcommand{\eqd}{\begin{eqnarray*}} % equation d�but
\newcommand{\eqf}{\end{eqnarray*}} % equation fin

\usepackage[pdftex]{hyperref}
\hypersetup{
  pdftitle={Asymptotic direction of random walks in Dirichlet environment},
  pdfauthor={Laurent Tournier},
  pdfpagemode=None
}
\usepackage{enumerate}
\newcommand{\drift}{{\vec{\Delta}}}
\newcommand{\vecu}{{\vec{u}}}

\newcommand{\defeq}{\mathrel{\mathop:}=}
\newcommand{\defqe}{=\mathrel{\mathop:}}

\newcommand{\alphach}{{\check{\alpha}}}

\newcommand{\sigmach}{{\check{\sigma}}}
\newcommand{\omegach}{{\check{\omega}}}
\newcommand{\ech}{{\check{e}}}
\newcommand{\Ech}{{\check{E}}}

\newcommand{\Gch}{{\check{G}}}
\newcommand{\much}{{\check{\mu}}}

\newcommand{\probab}{ {\mathcal P}}
\newcommand{\V}{ {\mathcal V}}
\newcommand{\dir}{ {\mathcal D}}
\newcommand{\vece}{ {\vec e}}
\newcommand{\HH}{{\mathcal H}}

\DeclareMathOperator{\area}{Area}

\begin{document}

%\fontfamily{pbk}\selectfont

%R\'efs mathscinet

{
  \noindent\LARGE \textbf{Asymptotic direction of random walks\\ in Dirichlet environment}

\vspace{.5cm}
\Large
\noindent\textit{Laurent Tournier\footnote{
Universit\'e Paris 13; CNRS UMR 7539; LAGA; 99 avenue J.-B. Cl\'ement, F-93430 Villetaneuse, France\\
\emph{E-mail: }{\tt tournier@math.univ-paris13.fr}\\This work was partly supported by the french ANR project Mememo2}}
}
\vspace{1cm}

\small
\paragraph{Abstract}
In this paper we generalize the result of directional transience from~\cite{SabotTournier10}. This enables us, by means of~\cite{Simenhaus}, \cite{Zerner-Merkl} and \cite{Bouchet} to conclude that, on $\Z^d$ (for any dimension~$d$), random walks in i.i.d.\ Dirichlet environment --- or equivalently oriented-edge reinforced random walks --- have almost surely an asymptotic direction equal to the direction of the initial drift, i.e.~$\frac{X_n}{\|X_n\|}$ converges to $\frac{E_o[X_1]}{\|E_o[X_1]\|}$ as $n\to\infty$, unless this drift is zero. In addition, we identify the exact value or distribution of certain probabilities, answering and generalizing a conjecture of~\cite{SabotTournier10}. 

\vspace{1cm}

\normalsize
\vspace{-.5cm}

\section{Introduction}

\paragraph{Presentation of the model.} Let $d\geq 1$, and denote by $(\vece_1,\ldots,\vece_d)$ the canonical basis of $\R^d$. Define the set
\[\mathcal{V}\defeq \{\vec{e}_1,-\vec{e}_1,\ldots,\vec{e}_d,-\vec{e}_d\}\subset\Z^d,\]
which will be used as possible steps, and assume we are given weights $\alpha_{\vec{e}}>0$, for $\vec{e}\in\mathcal{V}$. 

%We then define the weight of any oriented edge $(x,x+\vec{e})$ (where $x\in\Z^d$, $\vec{e}\in\mathcal{V}$) as
Consider now the directed graph $\Z^d$ whose oriented edges are the pairs $e=(x,y)$ such that $\vec{e}\defeq y-x$ is an element of $\mathcal{V}$, endowed with (initial) weight
\[\alpha_e\defeq \alpha_{\vec{e}}\] 
% $e=(x,x+\vec{e})$, where $x\in\Z^d$ and $\vec{e}\in\mathcal{V}$, have (initial) weight
%\[\alpha_e=\alpha_{(x,x+\vec{e})}\defeq\alpha_{\vec{e}}\]
and, for $x\in\Z^d$, define the law $P^{(\alpha)}_x$ of a random walk $(X_n)_{n\geq 0}$ on this graph in the following way: $P^{(\alpha)}_x$-a.s., $X_0=x$, and for every time $n\in\N$ and every edge $e$ starting at $X_n$, 
\begin{equation}\label{eq:def_oerrw}
P^{(\alpha)}_x\big( (X_n,X_{n+1})=e\big|X_0,\ldots,X_n\big)=\frac{\alpha_e+N_n(e)}{\sum\limits_{f : \underline{f}=X_n} \alpha_f+N_n(f)}
\end{equation}
where for an edge $e$ we let $e\defqe(\underline{e},\overline{e})$ and
\[N_n(e)\defeq\#\big\{0\leq i<n\ :\ (X_i,X_{i+1})=e\big\}.\]
Under $P_x^{(\alpha)}$, $(X_n)_{n\ge0}$ is called the \emph{oriented-edge reinforced random walk} (or more specifically the oriented-edge \emph{linearly} reinforced random walk) \emph{with initial weights $(\alpha_e)_e$, started at~$x$}. 

Due to the embedding of an independent Polya urn at each vertex and to a de Finetti property, this model admits an equivalent representation as a random walk in an i.i.d.~random environment given by Dirichlet random variables. Let us give a more precise statement. An \emph{environment} is an element $\omega=(\omega_x(\cdot))_{x\in\Z^d}$ of $\Omega\defeq\probab^{\Z^d}$ where $\probab$ is the simplex of probabilities on $\mathcal V$: 
\[\probab\defeq\Big\{(\omega(\vec e))_{\vec e\in\mathcal V}\,:\,\omega(\vec e)\geq 0,\ \sum_{\vec e\in\mathcal V}\omega(\vec e)=1\Big\}.\]
Given a starting point $x\in\Z^d$ and such an environment $\omega$, we may view $\omega$ as a set of transition probabilities (where $\omega_x(\vec e)$ is the transition probability from $x$ to $x+\vec e$) and define $P_x^\omega$ to be the law of the Markov chain starting at $x$ with transition probabilities given by $\omega$: for all $n\in\N$ and $\vec e\in\V$,  
\[P_x^\omega(X_{n+1}=X_n+\vec e\,|X_0,\ldots,X_n)=\omega_{X_n}(\vec e).\]
Finally, recall that the Dirichlet distribution $\dir^{(\alpha)}$  on $\probab$ with parameters $\alpha=(\alpha_{\vec e})_{\vec e\in\V}$ is the continuous probability distribution on $\probab$ given by
\[\dir^{(\alpha)}\defeq\frac{\Gamma(\sum_{\vec e\in\V}\alpha_{\vec e})}{\prod_{\vec e\in\V}\Gamma(\alpha_{\vec e})}\prod_{\vec e\in\V}p_\vece^{\alpha_\vece-1}\ {\rm d}\lambda(p),\]
where $\lambda$ is the Lebesgue measure on the simplex $\probab$, and denote by $\bP^{(\alpha)}=\big(\dir^{(\alpha)}\big)^{\otimes\Z^d}$ the law of an environment made of i.i.d.~Dirichlet marginals. Then we have the following identity (cf.~\cite{EnriquezSabot02} for instance):
\[P^{(\alpha)}_x(\cdot) = \int P^\omega_x(\cdot)\, {\rm d}\bP^{(\alpha)}(\omega).\] 

This representation constitutes the specificity of oriented-edge linear reinforcement and has been the starting point to prove several sharp results, in contrast to the still very partial understanding of either random walks in random environment or reinforced random walks in two and more dimensions. 

\paragraph{Context} Let us give a very brief account of the known results regarding transience before stating our result. 
We focus on the non-symmetric case, i.e.~when the weights are such that the \emph{mean drift}
\[\drift\defeq E^{(\alpha)}_o[X_1]= \frac{1}{\Sigma}\sum_{\vec{e}\in\mathcal{V}} \alpha_{\vec{e}}\,\vec{e}\]
is non-zero, where $\Sigma=\sum_{\vec{e}\in\mathcal{V}}\alpha_{\vec{e}}$. %In terms of random environment, this is the annealed drift at a vertex. 

In any dimension, Enriquez and Sabot \cite{EnriquezSabot06} gave the first result that was specific to Dirichlet environments, namely a sufficient ballisticity condition and bounds on the speed, later improved in~\cite{Tournier09}. On the other hand, non-ballistic cases are known to occur when weights are sufficiently small, due to the non-uniform ellipticity of the Dirichlet law (cf.~\cite{Tournier09}). Yet, under the only assumption of non-symmetry ($\drift\neq\vec 0$) ---~and thus for both ballistic and zero-speed cases~---, Sabot and the author~\cite{SabotTournier10} showed that  the random walk is transient in the direction of a basis vector with positive probability.  

In dimension $d\geq 3$, Sabot (in \cite{Sabot09}) proved the transience of these random walks (including in the symmetric case) and (in \cite{Sabot10}) gave a characterization of the ballistic regime (viz., ballisticity occurs when $\drift\neq\vec 0$ and the exit time from any edge is integrable, i.e.~$\forall\vece\in\V$, $2\Sigma-\alpha_\vece-\alpha_{-\vece}>1$). Finally, Bouchet~\cite{Bouchet} recently proved that the methods of~\cite{Sabot10} extend to non-ballistic cases up to an acceleration of the walk, which implies a 0-1 law for directional transience. 

%The last four papers exploit a property of stability of Dirichlet distributions under time reversal, which again proves particularly efficient in the present paper to obtain the following result.% generalizating the result of~\cite{SabotTournier10}. 

% The proof of the main theorem (and of the corollary in dimension $\leq2$) does not depend on the expression of the law of the environment but on its mere existence alone.  

%\[P^{(\alpha)}_x(\cdot) = \bE^{(\alpha)}[P_{x,\omega}(\cdot)],\]
%where $\bP^{(\alpha)}$ is the distribution on $\mathcal{V}^{\Z^d}$ such that the coordinates $\omega_z=(\omega_z(\vec{e}))_{\vec{e}\in\mathcal{V}}$, $z\in\Z^d$, are i.i.d.\ with distribution
%\[\mathcal{D}^{(\alpha)}\defeq\frac{\Gamma(\sum_{\vec{e}\in\mathcal{V}}\alpha_{\vec{e}})}{\prod_{\vec{e}\in\mathcal{V}}\Gamma(\alpha_{\vec{e}})}\Bigg(\prod_{\vec{e}\in\mathcal{V}}p_{\vec{e}}^{\alpha_{\vec{e}}-1}\Bigg)dp\]

%$\Sigma=\sum_{i=1}^d (\alpha_i+\alpha_{-i})$. 

%Let $\drift\defeq \Sigma \bE[X_1]=\sum_{i=1}^d (\alpha_i-\alpha_{-i})e_i$. 

\subsection{Directional transience and asymptotic direction}

\begin{thm}\label{thm:main}
Assume $\drift\neq \vec0$. For any $\vecu\in\R^d$ with rational slopes such that $\vecu\cdot\drift>0$,

 \[P_o^{(\alpha)}\left(X_n\cdot\vecu\limites{}{n}+\infty\right)>0.\] 
\end{thm}

This theorem was proved in~\cite{SabotTournier10} in the case when $\vecu=\vec{e}_i$. The interest in the present refinement lies in the corollary below, obtained by combining the theorem with the 0-1 laws of~\cite{Zerner-Merkl} ($d=2$) and of the recent~\cite{Bouchet} ($d\geq 3$) together with the main result of~\cite{Simenhaus}. (Details follow.) 

\begin{coroll}
Assume $\drift\neq\vec0$. Then, $P^{(\alpha)}_o$-a.s., the walk has an asymptotic direction that is given by the direction of the mean drift:
\[\frac{X_n}{\|X_n\|}\limites{}{n} \frac{\drift}{\|\drift\|},\qquad P^{(\alpha)}_o\text{-a.s.}\]
\end{coroll}

\paragraph{Remarks.}
\begin{itemize}
%	\item The explicit limit makes the corollary interesting also in cases already known to be ballistic, in that it gives the direction of the speed.  
%Although the existence of an asymptotic direction is much weaker than ballisticity, the above corollary provides an important information in the ballistic cases as well, namely the explicit direction of the speed. 
	%\item The proof of the theorem gives an explicit and simple lower bound: 
%\[P_o^{(\alpha)}\left( X_n\cdot\vecu\to+\infty \right) \geq \frac{E^{(\alpha)}_o[(X_1\cdot \vecu)_+]}{E_o^{(\alpha)}[X_1\cdot\vecu]}=\frac{\sum_{\vec e\in\mathcal{V}}\alpha_e(\vec e\cdot\vecu)_+}{\sum_{\vec e\in\mathcal{V}}\alpha_e \vec e\cdot\vecu}.\]
	\item In~\cite{EnriquezSabot06}, Enriquez and Sabot gave an expansion of the speed as $\gamma\to\infty$ when the parameters are $\alpha^{(\gamma)}_i:=\gamma\alpha_i$, and noticed that the second order was surprisingly colinear to the first one, i.e.\ to~$\drift$. This is not anymore a surprise given the above corollary; but this highlights the fact that the simplicity of the corollary comes as a surprise itself. Correlations between the transition probabilities at one site indeed affect the speed (cf.~for instance~\cite{Sabot04}), and thus the speed of a random walk in random environment is typically not expected to be colinear with the mean drift, if not for symmetry reasons. 
	\item The theorem does actually not depend on the graph structure of $\Z^d$ besides translation invariance, meaning that the result also holds for non nearest neighbour models: we may enable $\mathcal{V}$ to be any finite subset of $\Z^d$, and the proof is written in such a way that it covers this case. The same is true for the main results of~\cite{Bouchet} and~\cite{Simenhaus} with little modification, hence the corollary also generalizes in this way in dimension $\ge3$. The intersection property for planar walks used in~\cite{Zerner-Merkl} may however fail if jumps are allowed in such a way that the graph is not anymore planar. But if it is planar, then the proof carries closely. This includes in particular the case of the triangular lattice (by taking $\mathcal{V}=\{\pm\vec e_1,\pm\vec{e}_2,\pm(\vec e_1+\vec e_2)\}$). 
	\item Using the above-mentioned 0-1 laws, the probability in the theorem equals 1 and the rationality assumption is readily waived; the theorem was stated this way in order to keep its proof essentially contained in the present paper, in contrast to its corollary.  
	\item As a complement to the theorem, note that Statement (d) of Theorem 1.8 of~\cite{DrewitzRamirez} (and Lemma 4 of~\cite{Zerner-Merkl}) implies that in any dimension, when $\vecu\cdot\drift=0$, $P_o^{(\alpha)}$-a.s., $\limsup_n X_n\cdot\vecu=+\infty$ and $\liminf_n X_n\cdot\vecu=-\infty$. In dimension at least 3, this is showed in~\cite{Bouchet} as well. 
	\item Theorem 2 of~\cite{Bouchet} also implies the existence of a deterministic yet unspecified asymptotic direction in dimension at least 3. Further remarks regarding the derivation of the corollary from the theorem are deferred to the end of the proof. 
\end{itemize}

%(and for orthogonal direction, Drewitz-Ramirez ?)

\subsection{Identities} 
The proof of Theorem 1 goes through proving a lower bound on the probability that the walk never leaves the half-space $\{x:x\cdot\vecu\ge0\}$.  %$X_n\cdot\vecu$ is nonnegative for all times $n$.
In the next theorem, this lower bound is proved to be an equality. 

%In order to deal with general directions $\vecu$, 
Although the result admits a simple statement in some interesting cases (cf.~\eqref{eqn:conjecture} on page~\pageref{eqn:conjecture}), we need to introduce further notation to deal with general directions. 

Let $\vecu\in\R^d$ be a vector with rational slopes such that $\vecu\cdot\drift>0$. Due to periodicity, the ``discrete half-space'' $\{x\in\Z^d:x\cdot\vecu\ge0\}$ only has finitely many different entry points modulo translation. We shall denote by $\HH_0$ an arbitrary set of representative entry points and by $\mu\,(=\mu^{(\alpha,\vecu)})$ the probability measure on $\HH_0$ which makes $\mu(x)$ proportional to the total weight that enters vertex $x$ from outside of the previous half-space. Let us alternatively give more formal definitions. 

Up to multiplication by a positive number, we may assume $\vecu\in\Z^d$. We extend $\vecu$ into a basis $(\vecu,\vecu_2,\ldots,\vecu_d)$ chosen in such a way that $\vecu_i\perp\vecu$ and $\vecu_i\in\Z^d$ for all $i$. Since the ``discrete half-spaces''
\[\mathcal{A}_x\defeq\{y\in\Z^d\,:\,y\cdot\vecu\ge x\cdot\vecu\}\]
satisfy $\mathcal{A}_x=\mathcal{A}_{x\pm\vecu_i}$ for all $i\geq2$, $\mathcal A_x$ takes only finitely many different values when $x$ is in the ``discrete hyperplane''
\[\HH\defeq\big\{x\in\Z^d\,:\,\exists \vec e\in\mathcal{V},\,(x-\vec e)\cdot\vecu<0\leq x\cdot\vecu\},\]
namely for instance each of the values obtained when $x$ belongs to the finite set
\[\HH_0\defeq \HH\cap\big(\R_+\vecu+[0,\vecu_2)+\cdots+[0,\vecu_d)\big).\]
We then define the probability measure $\mu\,(=\mu^{(\alpha,\vecu)})$ on $\HH_0$ as follows: for all $x\in\HH_0$,% such that, for any $x\in\HH_0$, $\mu(x)$ is proportional to the sum of the weights of the edges entering in $x$ through $\mathcal{H}_0$: for some constant $Z$,
\begin{equation}\label{def:mu}
\mu(x)\defeq \frac{1}{Z}\sum_{\substack{\vec e\in\mathcal{V}:\\(x-\vec e)\cdot\vecu<0}}\alpha_{\vec e}
\end{equation}
where $Z$ is a normalizing constant (we have $Z>0$ because $\drift\cdot\vecu>0$). 

Let us also define a quenched analogue to $\mu$. We enlarge $\Omega$ by adding a component $\omega(\partial,\cdot)$ to each environment $\omega$, where $\omega(\partial,\cdot)$ is a probability distribution on $\HH_0$. And we extend $\bP^{(\alpha)}$ so that $\omega(\partial,\cdot)$ is independent of $(\omega_x(\cdot))_{x\in\Z^d}$ and follows a Dirichlet distribution of parameters
\[%\Bigg(\begin{array}{cc}\sum&  \!\!\!\alpha_{\vece}\\\substack{\vec e\in\mathcal{V}:\\(x-\vec e)\cdot\vecu<0}&\end{array}\Bigg)_{x\in\HH_0} \quad 
\quad \bigg(\sum_{\substack{\vec e\in\mathcal{V}:\\(x-\vec e)\cdot\vecu<0}}\alpha_{\vec e}\Bigg)_{x\in\HH_0}.\]
Note that, for $x\in\HH_0$, \[\mu(x)=\int \omega(\partial,x)\,{\rm d}\bP^{(\alpha)}(\omega).\] 

\begin{thm} \label{thm:identities}
Assume that $\drift\neq\vec0$ and $\vecu$ is a vector of $\R^d$ with rational slopes such that $\vecu\cdot\drift>0$. Then the following identity holds:
\begin{equation}\label{eq:identity}
P^{(\alpha)}_\mu(\forall n\geq 0,\,X_n\cdot\vecu\geq 0)= 1-\frac{E^{(\alpha)}_o\big[(X_1\cdot\vecu)_-\big]}{E^{(\alpha)}_o\big[(X_1\cdot\vecu)_+\big]}.
\end{equation}
For the walk inside the cylinder
\[C\defeq \Z^d\Big\slash(\Z\vecu_2+\cdots+\Z\vecu_d),\]
the previous identity also holds, as well as the following one involving distributions: 
\begin{equation}\label{eq:identity_cylinder}
\mathcal{L}_{\bP^{(\alpha)}}\left(P^{\omega}_{\omega(\partial,\cdot)}(\forall n\geq 0,\,X_n\cdot\vecu\geq 0)\right)=%{\rm Beta}\Big(\sum_{x\in\HH_0, \vece\in\V}\signe(\vece\cdot\vecu)\alpha_\vece,\ \sum_{x\in\HH_0, \vece\in\V}\indic_{\{\vecu\cdot\vece<0\}}\alpha_\vece\Big)
 %{\rm Beta}\Big(\sum_{x\in\HH_0}\sum_{\substack{\vece\in\V:\\(x-\vece)\cdot\vecu<0}}\alpha_\vece-\sum_{\substack{\vece\in\V:\\(x+\vece)\cdot\vecu<0}}\alpha_\vece,\ \sum_{x\in\HH_0}\sum_{\substack{\vece\in\V:\\(x+\vece)\cdot\vecu<0}}\alpha_\vece\Big)
{\rm Beta}\Big(\sum_{\substack{x\in\HH_0,\vece\in\V:\\(x\pm\vece)\cdot\vecu<0}}\mp\alpha_\vece,\ \sum_{\substack{x\in\HH_0,\vece\in\V:\\(x+\vece)\cdot\vecu<0}}\alpha_\vece\Big)
\end{equation}
where $\mathcal{L}_{\bP}(X)$ denotes the law under probability $\bP$ of a random variable $X$ and ${\rm Beta}(\cdot,\cdot)$ is the classical Beta distribution.  
\end{thm}

\paragraph*{Remarks.} 
\begin{itemize}
	\item The distribution of $X_1$ under $P^{(\alpha)}_o$ is simply given by the initial weights, hence \eqref{eq:identity} is fully explicit. This also follows from taking the expectation of the law in~\eqref{eq:identity_cylinder} (the expectation of ${\rm Beta}(a,b)$ is $\frac a{a+b}$). 
	\item The case $\vecu=\vec e_1$ with nearest-neighbour jumps admits a simple expression. Indeed, $\HH_0=\{0\}$ hence the results read as follows: if $\alpha_1>\alpha_{-1}$, 
\begin{equation}\label{eqn:conjecture}
P^{(\alpha)}_o(\forall n\geq 0,\,X_n\cdot\vece_1\geq 0)= 1-\frac{\alpha_{-1}}{\alpha_1}
\end{equation}
(as conjectured in~\cite{SabotTournier10}), and on the cylinder $\Z\times\mathbb{T}$ with $\mathbb T=\Z^{d-1}/(\Z\vec v_2+\cdots+\Z\vec v_d)$ for some basis $(\vec v_2,\ldots,\vec v_d)$ of $\R^{d-1}$ with integer coordinates, 
\[
\mathcal{L}_{\bP^{(\alpha)}}\left(P^{\omega}_{\omega(\partial,\cdot)}(\forall n\geq 0,\,X_n\cdot\vece_1\geq 0)\right)= {\rm Beta}(\alpha_1-\alpha_{-1},\alpha_1)
\]
where, under $\bP^{(\alpha)}$, $\omega(\partial,\cdot)$ follows a Dirichlet distribution on $\{0\}\times\mathbb T$ with all parameters equal to $\alpha_1$. 
	\item In dimension 1, the identities already follow from~\cite{SabotTournier10} in a simple way. Note that they are not trivial even in this case: the quenched identity actually dates back to~\cite{Chamayou-Letac} where it was proved in a completely different way. 
	\item Mild variations of the proof also provide other identities, as for instance
\[E^{(\alpha)}_\mu\left[\widetilde T_0^\vecu\middle|\widetilde T_0^\vecu<\infty\right]=E^{(\alpha)}_\mu[T^\vecu_0]+1-\frac{E^{(\alpha)}_o\big[(X_1\cdot\vecu)_+\big]}{E^{(\alpha)}_o\big[(X_1\cdot\vecu)_-\big]}\]
or, for all $L\in\N$ such that $L\|\vecu\|>\|\vece\|$ for all $\vece\in\V$,
\begin{equation}\label{eq:identity_slabs}
\frac{P^{(\alpha)}_\mu(\widetilde{T}^\vecu_0<T^\vecu_L)}{P^{(\alpha)}_\mu(T^\vecu_0<\widetilde{T}^\vecu_{-L})}= \frac{E^{(\alpha)}_o\big[(X_1\cdot\vecu)_-\big]}{E^{(\alpha)}_o\big[(X_1\cdot\vecu)_+\big]}
\end{equation}
where, for $L\in\Z$, we defined the projected hitting times
\[T^\vecu_L\defeq\inf\{n\,:\,X_n\cdot\vecu>L\|\vecu\|^2\}\quad\text{and}\quad
\widetilde{T}^\vecu_L\defeq\inf\{n\,:\,X_n\cdot\vecu<L\|\vecu\|^2\}.\]
\end{itemize}

\section{Proof of Theorem~\ref{thm:main}: Directional transience}

The proof, like~\cite{SabotTournier10}, uses a time reversal property from~\cite{Sabot09} (re-proved in a more probabilistic way in \cite{SabotTournier10}). To keep the present proof more self-contained, and for the sake of introducing some notation, we recall the following very elementary (yet powerful) lemma that sums up the only aspect of this property that we will use later. This is Lemma~1 of~\cite{SabotTournier10}.

\begin{lemme}\label{lem:cycle}
Let $G=(V,E)$ be a directed graph, endowed with positive weights $(\alpha_e)_{e\in E}$. We denote by $\Gch=(V,\Ech)$ its reversed graph, i.e.~$\Ech\defeq\{\ech\defeq(\overline{e},\underline{e})\,:\,e=(\underline{e},\overline{e})\in E\}$, endowed with the weights $\alphach_\ech\defeq\alpha_e$. Assume that $\mathrm{div}(\alpha)=0$, i.e.,~for every $x\in V$,
\[\alpha_x\defeq \sum_{e\,:\,\underline{e}=x}\alpha_e = \sum_{e\,:\,\overline{e}=x}\alpha_e\defqe \alphach_x.\] 
Then, for any closed path $\sigma=(x_0,x_1,\ldots,x_{n-1},x_0)$ in $G$, letting $\sigmach\defeq(x_0,x_{n-1},\ldots,x_1,x_0)$ denote its reverse (in $\Gch$), we have
\[P_{x_0}^{(\alpha)}\big( (X_0,\ldots,X_n)=\sigma \big)=P_{x_0}^{(\alphach)}\big( (X_0,\ldots,X_n)=\sigmach \big),\]
where the laws of oriented-edge reinforced random walks on $G$ or $\Gch$ are defined as in~\eqref{eq:def_oerrw}.  
%\[P^{(\alpha)}(X_0=x_0,X_1=x_1,\ldots,X_{n-1}=x_{n-1},X_n=x_0)=P^{(\alphach)}(X_0=x_0,X_1=x_{n-1},\ldots,X_{n-1}=x_1,X_n=x_0).\]
\end{lemme}

\begin{proof}
From the definition of $P_{x_0}^{(\alpha)}$ we get
\[P_{x_0}^{(\alpha)}\left( (X_0,\ldots,X_n)=\sigma \right)=\frac{\prod_{e\in E} \alpha_e(\alpha_e+1)\cdots (\alpha_e+n_e(\sigma)-1)}{\prod_{x\in V} \alpha_x(\alpha_x+1)\cdots (\alpha_x+n_x(\sigma)-1)},\]
where $n_e(\sigma)$ (resp.~$n_x(\sigma)$) is the number of crossings of the oriented edge $e$ (resp. the number of visits of the vertex $x$) in the path $\sigma$. Cyclicity gives $n_e(\sigma)=n_\ech(\sigmach)$ and $n_x(\sigma)=n_x(\sigmach)$ for all $e\in E,x\in V$. Furthermore we have by assumption $\alphach_x=\alpha_x$ for every vertex $x$, and by definition $\alpha_e=\alphach_\ech$ for every edge $e$. This shows that the previous product matches the similar product with $\Ech$, $\alphach$ and $\sigmach$ instead of $E$, $\alpha$ and $\sigma$, hence the lemma. 
\end{proof}

Let us turn to the proof of Theorem~\ref{thm:main}. 
Assume $\drift\neq\vec0$, and let $\vecu$ be a vector with rational slopes such that $\drift\cdot\vecu>0$.

%Let us first note that it suffices to prove the second part of the theorem for vectors $\vecu\in\Q^d$. Indeed, any $\vecv$ satisfying the condition is a convex combination of vectors $\vecu_i\in\Q^d$ satisfying the condition as well, and $\{\ell\in\R^d\setminus\{0\}|X_n\cdot\ell\to_n+\infty\}$ is convex. ** INSUFFISANT : et si chaque $\vecu_i$ correspondait \`a un \'ev\'enement disjoint des autres ??

% NB : l'ensemble $\{\ell\in\R^d|X_n\cdot\ell\to+\infty\}$ peut \^etre r\'eduit \`a un point, par exemple si $X_n=(n,(-1)^n n^2,0,...)$

We make use of the notations introduced before Theorem~\ref{thm:identities}. As in the introduction, up to multiplication by a constant, we may assume that $\vecu\in\Z^d$, and also that $\|\vecu\|\geq\|\vec e\|$, $\forall \vec e\in\mathcal V$ (we may have $\|\vec e\|>1$, cf.~the second remark after the corollary). Remember that $(\vecu,\vecu_2,\ldots,\vecu_d)$ is a basis such that $\vecu_i\in\Z^d$ and $\vecu_i\perp\vecu$ for all $i$. 

Let us consider the event $D\defeq\{\forall n\geq 0,\,X_n\cdot\vecu\geq0\}$, and define a finite graph that will enable us to bound $P^{(\alpha)}_\mu(D)\defeq\sum_{x\in\HH_0}\mu(x)P^{(\alpha)}_x(D)$ from below. 

Let $N,L\in\mathbb{N}^*$. We first consider the cylinder
\begin{align*}
C_{N,L}
    & \defeq \Big\{x\in\Z^d\,:\,0\leq x\cdot\vecu\leq L\|\vecu\|^2\Big\}\Big\slash(N\Z\vecu_2+\cdots+N\Z\vecu_d),
\end{align*}
i.e.~the slab $\{0\leq x\cdot\vecu\leq L\|\vecu\|^2\}\cap\Z^d$ where vertices that differ by $N\vecu_i$ for some $i\in\{2,\ldots,d\}$ are identified. % (the length of the cylinder actually doesn't have to be a multiple of $\|\vecu\|$ for the following to hold). 
Let $\mathcal{R}$ denote its ``right'' end, i.e.
\[\mathcal{R}\defeq\Big\{x\in\Z^d\,:\,\exists\vec e\in\mathcal V,\, x\cdot\vecu\leq L\|\vecu\|^2<(x+\vec e)\cdot\vecu\Big\}\Big\slash(N\Z\vecu_2+\cdots+N\Z\vecu_d)\subset C_{N,L}\]
%\[\mathcal{R}\defeq\Big\{x\in\Z^d\,:\,\exists y\in\Z^d, 0\leq x\cdot\vecu\leq L\|\vecu\|^2<y\cdot\vecu\text{ and }\|y-x\|=1\Big\}\Big\slash(N\Z\vecu_2+\cdots+N\Z\vecu_d)\subset C_{N,L}\]
(note that the inclusion holds for small $L$ due to the constraint $\|\vecu\|\geq\|\vec e\|$) and similarly $\mathcal{L}\subset C_{N,L}$ for the ``left'' end. We may now define the finite graph $G_{N,L}$ (refer to Figure~\ref{fig:grid2d} for an example in $\Z^2$). Its vertex set is
\[V_{N,L}\defeq C_{N,L}\cup\{R,\partial\},\]
where $R$ and $\partial$ are new vertices, and the edges of $G_{N,L}$ are of the following types:
\begin{enumerate}[a)]\setlength{\itemsep}{0pt}
	\item edges induced by those of $\Z^d$ inside $C_{N,L}$;%, i.e.~from any $x\in C_{N,L}$ to $x+\vec e$ where $\vec e\in\mathcal V$, provided $x+\vec e\in C_{N,L}$;
    \item edges from (resp.~to) the vertices of $\mathcal{L}$ to (resp.~from) $\partial$, corresponding to the edges of $\Z^d$ exiting (resp.~entering) the cylinder ``through the left end'';
	\item edges from (resp.~to) the vertices of $\mathcal{R}$ to (resp.~from) $R$, corresponding to the edges of $\Z^d$ exiting (resp.~entering) the cylinder ``through the right end'';
    \item a new edge from $R$ to $\partial$.
\end{enumerate}
Note that in b) and c) several edges may connect two vertices, and that in d) no edge goes from $\partial$ to $R$. 
We also introduce weights $\alpha^{N,L}_e$ on the edges of $G_{N,L}$ as follows (invoking the translation invariance of the weights in $\Z^d$):
\begin{itemize}
    \item edges defined in a), b) and c) have the weight of the corresponding edge in $\Z^d$;
    \item the edge from $R$ to $\partial$ has weight
\[\alpha^{N,L}_{(R,\partial)}\defeq\Bigg(\sum_{\substack{x\in\mathcal{R},\,\vec e\in\mathcal{V}:\\x+\vec e\notin C_{N,L}}}\alpha_{\vec e}\Bigg) -\Bigg( \sum_{\substack{x\in\mathcal{L},\,{\vec e}\in\mathcal{V}:\\x+\vec e\notin C_{N,L}}}\alpha_{\vec e}\Bigg).\]
%\[\alpha^{N,L}_{(R,\partial)}\defeq\sum_{\substack{x\in\mathcal{R},\,e\in\mathcal{V}:\\x+e\notin C_{N,L}}}(\alpha_e-\alpha_{-e}).\]
\end{itemize}

By construction% (and because of the periodicity of the weights in direction $\vecu$)
, we have $\mathrm{div}\,\alpha^{N,L}=0$. The main point to check however is that $\alpha^{N,L}_{(R,\partial)}$ is positive. 

Due to periodicity, $\mathcal L$ (and $\mathcal R$) decomposes into $N^{d-1}$ subsets which are translations of $\HH_0$ and we have
\[\alpha^{N,L}_{(R,\partial)}%=\sum_{\vec e\in\mathcal V}\Bigg(\sum_{\substack{x\in\mathcal R:x+\vec e\notin C_{N,L}}}1-\sum_{\substack{x\in\mathcal R:x-\vec e\notin C_{N,L}}}1\Bigg)\alpha_{\vec e}
=N^{d-1}\area(\vecu_2,\cdots,\vecu_d)\sum_{\vec e\in\mathcal V}\big(\Phi_\vecu(\vec e)-\Phi_{-\vecu}(\vec e)\big)\alpha_{\vec e}\] 
where $\area(\vecu_2,\ldots,\vecu_d)=\frac{|\det(\vecu,\vecu_2,\ldots,\vecu_d)|}{\|\vecu\|}$ is the $(d-1)$-volume of the hypersurface $[0,\vecu_2]+\cdots+[0,\vecu_d]$ and $\Phi_\vecu(\vec e)$ is the flux of $\vec e$ through the oriented hyperplane $\vecu^\perp$:
\begin{align*}
\Phi_{\vecu}(\vec e)
	& \defeq \frac1{N^{d-1}\area(\vecu_2,\cdots,\vecu_d)}\#\big\{x\in\mathcal{R}\,:\, x+\vec e\notin C_{N,L}\big\}\\
	& \ = \frac1{\area(\vecu_2,\cdots,\vecu_d)}\#\bigg(\big\{x\in\Z^d\,:\, x\cdot\vecu\leq 0 < (x+\vec e)\cdot\vecu\big\}\bigg/\big(\Z\vecu_2+\cdots+\Z\vecu_d\big) \bigg). 
\end{align*}
Clearly $\Phi_\vecu(\vec e)$ is zero if $\vecu\cdot\vec e\leq 0$ and otherwise it is a simple geometric fact that the last cardinality above equals the volume of the parallelotope on the vectors $\vec e, \vecu_2,\ldots, \vecu_d$. Indeed, this cardinality is also the number of lattice points in the torus $\R^d\big/(\Z\vec e+\Z\vecu_2+\cdots+\Z\vecu_d)$, and this torus can be partitioned into the unit cubes $x+[0,1)^d$ indexed by the lattice points $x$ in it. Hence in any case 
\[\Phi_\vecu(\vec e)=\frac{{\rm Vol}(\vece,\vecu_2,\ldots,\vecu_d)}{\area(\vecu_2,\ldots,\vecu_d)}\indic_{(\vecu\cdot\vece>0)}=\Big(\frac\vecu{\|\vecu\|}\cdot\vec e\Big)_+.\]
This gives
\begin{align*}
\alpha^{N,L}_{(R,\partial)}
	& =N^{d-1}\area(\vecu_2,\ldots,\vecu_d)\sum_{\vec e\in\mathcal V}\big(\bigg(\frac\vecu{\|\vecu\|}\cdot\vec e\bigg)_+-\bigg(-\frac\vecu{\|\vecu\|}\cdot\vec e\bigg)_+\big)\alpha_{\vec e}\\
	& =N^{d-1}\area(\vecu_2,\ldots,\vecu_d)\sum_{\vec e\in\mathcal V}\bigg(\frac\vecu{\|\vecu\|}\cdot\vec e\bigg)\alpha_{\vec e}\\
	& = N^{d-1}\area(\vecu_2,\ldots,\vecu_d)\frac\vecu{\|\vecu\|}\cdot\Sigma \drift
\end{align*}
therefore finally $\alpha^{N,L}_{(R,\partial)}>0$ since $\vecu\cdot\drift>0$, as expected. 

%The proof of this geometrical fact is deferred to Section x where a general statement is given. In dimension 2 and for nearest neighbour steps, a convincing ``visual proof'' can be deduced from looking at Figure 1: along the right hand side of the cylinder, there is one horizontal exiting edge $e_1$ for each line and one vertical exiting edge $e_2$ for each column hence a total vector of $ne_1+me_2$ if there are $n$ lines and $m$ columns, and this vector is orthogonal to the directing vector $me_1-ne_2$ of the right hand side, hence it is colinear to $u$. 

%This definition furthermore ensures that $\mathrm{div}\,\alpha^{N,L}=0$. 

NB. The above computation also shows that, introducing a new notation, 
\[\alpha^{N,L}_{(\mathcal{L},\partial)}\defeq\sum_{x\in\mathcal{L}}\alpha^{N,L}_{(x,\partial)}=\sum_{\substack{x\in\mathcal{L},\,\vec e\in\mathcal{V}:\\x-\vec e\notin C_{N,L}}}\alpha_{\vec e}=N^{d-1}\area(\vecu_2,\ldots,\vecu_d)\sum_{\vec e\in\mathcal V}\bigg(-\frac\vecu{\|\vecu\|}\cdot\vec e\bigg)_+\alpha_{\vec e},\]
hence in particular
\begin{equation}\label{eqn:ratio}
\frac{\alpha^{N,L}_{(\mathcal{L},\partial)}}{\alpha^{N,L}_{({R},\partial)}}=\frac{\sum_{\vec e\in\mathcal V}\big(-\vecu\cdot\vec e\big)_+\alpha_{\vec e}}{\sum_{\vec e\in\mathcal V}\big(\vecu\cdot\vec e\big)\alpha_{\vec e}}=\frac{E^{(\alpha)}_o\big[(X_1\cdot\vecu)_-\big]}{E^{(\alpha)}_o\big[X_1\cdot\vecu\big]}.
\end{equation}

\begin{figure}
\begin{center}
\resizebox{12cm}{!}{\includegraphics{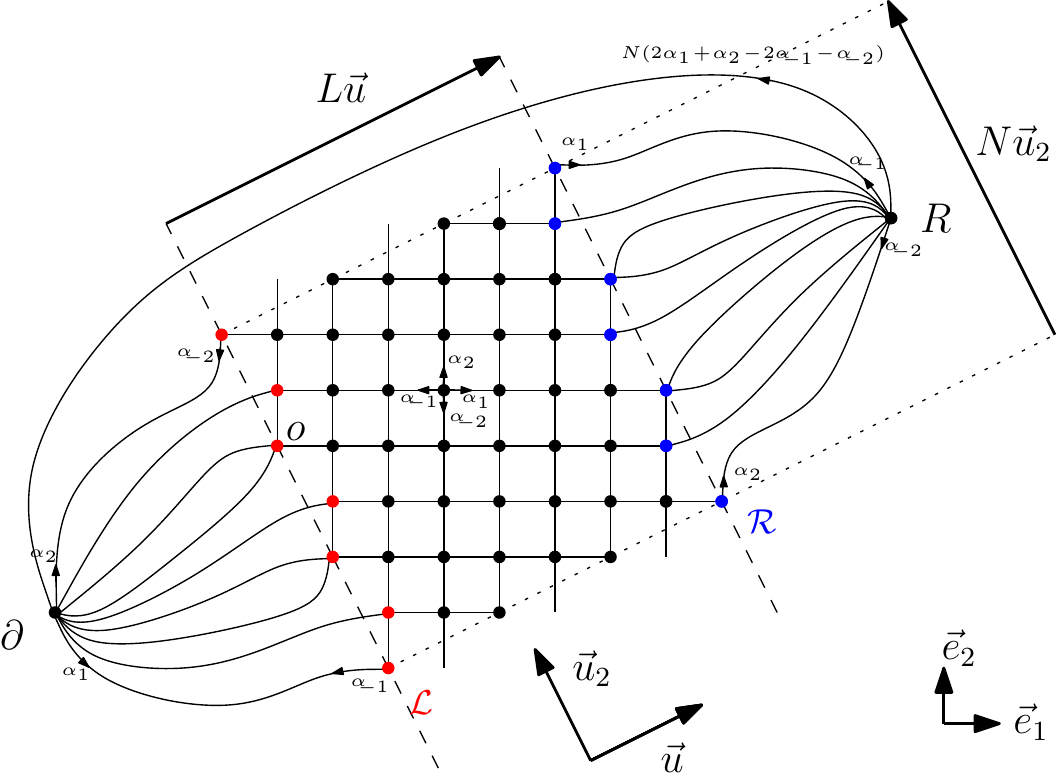}}
\end{center}
\caption{Graph $G_{N,L}$ for $\vecu=2\vec e_1+\vec e_2$ (boundary conditions in direction $\vecu_2$ are periodic)}
\label{fig:grid2d}
\end{figure}

%The maps made of translations by any of $\vecu_2,\ldots,\vecu_d$ in $C_{N,L}$ and mapping $R$ and $\partial$ to themselves are weighted-graph isomorphisms, since $\vecu_i\perp\vecu$ and $\vecu_i\in\Z^d$. For this reason, and translation invariance of the weights in direction $\vecu$ as well,  
%\[\alpha^{N,L}_{(R,\partial)}=N^{d-1}\alpha^{1,0}_{(R,\partial)},\]
%and similarly (introducing a new notation)
%\[\alpha^{N,L}_{(\mathcal{L},\partial)}\defeq\sum_{x\in\mathcal{L}}\alpha^{N,L}_{(x,\partial)}=\sum_{\substack{x\in\mathcal{L},\,e\in\mathcal{V}:\\x+e\notin C_{N,L}}}\alpha_e=N^{d-1}\alpha^{1,0}_{(\mathcal{L},\partial)}.\]
%Translation by vector $L\vecu$ also gives $\alpha^{N,L}_{(R,\partial)}=\alpha^{N,L}_{(\mathcal{R},R)}-\alpha^{N,L}_{(\mathcal{L},\partial)}$ where like above $\alpha_{(\mathcal{R},R)}\defeq\sum_{x\in\mathcal{R}}\alpha^{N,L}_{(x,R)}$. 

If the walk starts from $X_0=\partial$, then we have $X_1= Z\ ({\rm mod}\ \vecu_2,\ldots,\vecu_d)$ where $Z$ has law $\mu$ (defined in~\eqref{def:mu}). Using translation invariance with respect to vectors $\vecu_2,\ldots,\vecu_d$, and the fact that, starting at $\partial$, the event $\{H_R<H^+_\partial\}$ (where $H$ stands for hitting time and $H^+$ for positive hitting time) only depends on the walk before its first return in $\partial$ --- and thus not on the reinforcement of the very first edge --- we deduce, considering $\mu$ as a law on (a subset of) $\mathcal{L}$,
\begin{align}
P_{\mu}^{(\alpha^{N,L})}(H_R<H_\partial)
	& =P_\partial^{(\alpha^{N,L})}(H_R\circ\theta_1<H_\partial\circ\theta_1)=P_\partial^{(\alpha^{N,L})}(H_R<H^+_\partial)  \label{eq:change_x0}
\end{align}
(using $\theta$ to denote time shift) and thus 
\[
P_{\mu}^{(\alpha^{N,L})}(H_R<H_\partial)
 \geq P_\partial^{(\alpha^{N,L})}(X_{H_\partial-1}=R). 
\]
The last event is the probability that the walk follows a cycle in a given family (namely cycles going through $\partial$ only once and containing the edge $(R,\partial)$). Applying Lemma~\ref{lem:cycle} to every such cycle and summing up, we get (using~\eqref{eqn:ratio} for the last equality)
\begin{align*}
P_\partial^{(\alpha^{N,L})}(X_{H_\partial-1}=R)
	& =P_\partial^{(\alphach^{N,L})}(X_1=R)\\
	& =\frac{\alpha^{N,L}_{(R,\partial)}}{\alpha^{N,L}_{(R,\partial)}+\alpha^{N,L}_{(\mathcal{L},\partial)}}\\
	& = \frac{E^{(\alpha)}_o\big[X_1\cdot\vecu\big]}{E^{(\alpha)}_o\big[(X_1\cdot\vecu)_+\big]}.
	%& =\frac{\alpha^{1,0}_{(R,\partial)}}{\alpha^{1,0}_{(R,\partial)}+\alpha^{1,0}_{(\mathcal{L},\partial)}}.
\end{align*}
This lower bound is positive and uniform with respect to $L$ and $N$.  We may rewrite the result as
\[P_{\mu}^{(\alpha^{N,L})}(H_R<H_\partial)\geq 1-\frac{E^{(\alpha)}_o\big[(X_1\cdot\vecu)_-\big]}{E^{(\alpha)}_o\big[(X_1\cdot\vecu)_+\big]}.\]
%where incidentally, when $\vecu=\drift$, the last quantity can be seen to be equal to $1-\frac{\sum_{e\in\mathcal{V}}\alpha_e(\drift\cdot e)_-}{\sum_{e\in\mathcal{V}}\alpha_e(\drift\cdot e)_+}$.
%This expression has the advantage of depending continuously of $\drift$. Et alors ? vu que pour une pente irrationnelle le point de d\'epart pose probl\`eme\ldots

Letting $N$ and then $L$ go to infinity as in \cite{SabotTournier10} (applied to each of the finitely many possible values of $X_0$ in $\HH_0$), we get
\[P^{(\alpha)}_\mu(\forall n\geq 0,\,X_n\cdot\vecu\geq 0)\geq 1-\frac{E^{(\alpha)}_o\big[(X_1\cdot\vecu)_-\big]}{E^{(\alpha)}_o\big[(X_1\cdot\vecu)_+\big]}\]
hence, by translation invariance of $P^{(\alpha)}_o$ and Kalikow's 0-1 law (and Lemma 4 of~\cite{Zerner-Merkl}, showing that the walk cannot stay in a slab), 
\begin{align*}
P^{(\alpha)}_o(X_n\cdot\vecu\limites{}{n}+\infty)
	 =P^{(\alpha)}_\mu(X_n\cdot\vecu\limites{}{n}+\infty)
	&\ge P^{(\alpha)}_\mu(\forall n\geq 0,\,X_n\cdot\vecu\geq 0)\\
	& \geq  1-\frac{E^{(\alpha)}_o\big[(X_1\cdot\vecu)_-\big]}{E^{(\alpha)}_o\big[(X_1\cdot\vecu)_+\big]}>0.
\end{align*}
This is the content of Theorem~\ref{thm:main}.

%thus there is $x\in\Z^d$ such that $P^{(\alpha)}_x(\forall n\geq 0,\,X_n\cdot\vecu\geq 0)>0$, hence $P^{(\alpha)}_x(X_n\cdot\vecu\to_n+\infty)>0$ by Kalikow's 0-1 law. The latter event does of course not depend on $x$, hence the theorem. 

\section{Proof of the Corollary: Asymptotic direction}

Recall from the introduction that oriented-edge reinforced random walks are also random walks in Dirichlet environment. Due to the 0-1 law of Zerner and Merkl~\cite{Zerner-Merkl} (cf.~also~\cite{Zerner}) in dimension 2 (for random walks in elliptic random environment), and of Bouchet~\cite{Bouchet} in dimension at least 3 (for random walks in Dirichlet environment), the result of Theorem~\ref{thm:main} turns into: for any $\vec u\in\R^d$ with rational slopes and such that $\vec u\cdot\drift>0$, 
\begin{equation}\label{eqn:transience}
X_n\cdot\vecu\limites{}{n}+\infty,\qquad P_o^{(\alpha)}-a.s.
\end{equation}
Note that the set of directions $\vecu\in\R^d$ such that~\eqref{eqn:transience} holds also has to be convex, therefore it contains the half-space $\{\vecu\in\R^d\,:\,\vecu\cdot\drift>0\}$. 

By Theorem~1 of \cite{Simenhaus}, there exists a direction $\vec \nu\in\mathbb S^{d-1}$ such that
\[\frac{X_n}{\|X_n\|}\limites{}{n}\vec \nu,\qquad P_o^{(\alpha)}-a.s.\]
On the other hand, this direction satisfies $\vec \nu\cdot\vecu\geq0$ for every $\vecu$ that satisfies~\eqref{eqn:transience}, hence in particular for every $\vecu$ such that $\drift\cdot\vecu>0$. This fully characterizes $\vec\nu$, which therefore has to be
\[\vec\nu=\frac{\drift}{\|\drift\|}.\]

\paragraph{Remarks}
\begin{itemize}
	\item Before learning about the article~\cite{Bouchet}, a former (private) version of the present paper gave a weaker result in dimension at least 3, namely that an asymptotic direction exists, although it remained unidentified, and possibly random (two-valued). Indeed, by the 0-1 law of Kalikow (in its elliptic version proved in~\cite{Zerner-Merkl}) and Theorem 1.8 of~\cite{DrewitzRamirez}, there exists $\vec \nu\in\mathbb{S}^{d-1}$ and an event $A$ such that, almost-surely,
\[\frac{X_n}{\|X_n\|}\limites{}{n} (\indic_A-\indic_{A^c})\vec \nu\]
but identifying $\vec\nu$ from Theorem~\ref{thm:main} is hindered by the restriction to rational slopes due to the possible non-convexity of the set of directions $\vecu$ of transience (i.e.\ satisfying the theorem).
	\item In dimension at least 3, since~\cite{Bouchet} already proves the existence of an asymptotic direction, an alternative derivation of the corollary without~\cite{Simenhaus} would consist in using Theorem~\ref{thm:main} in the proof of Theorem~2 of~\cite{Bouchet} instead of referring to~\cite{SabotTournier10}.
%	\item ref \`a ce qu'\'ecrit Simenhaus
\end{itemize}

\newpage
\section{Proof of Theorem~\ref{thm:identities}: Identities}

%Although the quenched statement of Theorem~\ref{thm:identities} simply implies the annealed one, for ease of exposition we first present a proof of the annealed identity before we turn to its quenched version. 

\subsection{Annealed identity}

Let $N,L\in\N^*$. Let us make use of the definitions involved in the proof of Theorem~\ref{thm:main}, in particular the graph $G_{N,L}$, and apply Lemma~\ref{lem:cycle} to a different family of cycles. 

As in~\eqref{eq:change_x0}, we have
\begin{align*}
P_{\mu}^{(\alpha^{N,L})}(H_R<H_\partial) =P_\partial^{(\alpha^{N,L})}(H_R<H^+_\partial) = 1-P_\partial^{(\alpha^{N,L})}(H^+_\partial<H_R).
\end{align*}
The last event is the probability that the walk follows a cycle in a given family (namely cycles that pass through $\partial$ exactly once and don't visit $R$). %Note that none of these cycles visits the vertex $R$ hence the additional edge $(R,\partial)$ plays no role. 
Note that this set of cycles is globally invariant by change of orientation. Thus, applying Lemma~\ref{lem:cycle} to every such cycle and summing up, we get
\[P_\partial^{(\alpha^{N,L})}(H^+_\partial<H_R)=P_\partial^{(\alphach^{N,L})}(H^+_\partial<H_R).\]
The edge $(\partial,R)$ is in $\Gch$, hence we may decompose the event on the right as follows: the first step is different from $R$, and then the walk comes back to $\partial$ before reaching $R$. Since the edge $(\partial,X_1)$ is not involved in the second part, these events are independent and we have
\begin{align*}
P_\partial^{(\alphach^{N,L})}(H^+_\partial<H_R)
	& = P_\partial^{(\alphach^{N,L})}(X_1\neq R)P_\much^{(\alphach^{N,L})}(H_\partial<H_R),
\end{align*}
where $\much$ is defined like $\mu$ with respect to $\alphach$ instead of $\alpha$. First, using~\eqref{eqn:ratio} for the last equality,
\[P_\partial^{(\alphach^{N,L})}(X_1\neq R) = 1-\frac{\alpha_{(R,\partial)}}{\alpha_{(R,\partial)}+\alpha_{(\mathcal L,\partial)}}=\frac{E^{(\alpha)}_o\big[(X_1\cdot\vecu)_-\big]}{E^{(\alpha)}_o\big[(X_1\cdot\vecu)_+\big]}.\]
%On the other hand, 
%\[P_\much^{(\alphach^{N,L})}(H_\partial<H_R)=1-P_\much^{(\alphach^{N,L})}(H_R<H_\partial),\]
Gathering everything, we obtain
\[P_{\mu}^{(\alpha^{N,L})}(H_R<H_\partial) = 1 -\frac{E^{(\alpha)}_o\big[(X_1\cdot\vecu)_-\big]}{E^{(\alpha)}_o\big[(X_1\cdot\vecu)_+\big]}\Big(1-P_\much^{(\alphach^{N,L})}(H_R<H_\partial)\Big).\] 
Arguing like for Theorem~\ref{thm:main} (i.e.\ cf.~\cite{SabotTournier10}), we let $N$, then $L$ go to infinity and get
\[P_{\mu}^{(\alpha)}(\forall n\geq 0,\,X_n\cdot\vecu\geq 0) = 1 -\frac{E^{(\alpha)}_o\big[(X_1\cdot\vecu)_-\big]}{E^{(\alpha)}_o\big[(X_1\cdot\vecu)_+\big]}\Big(1-P_\much^{(\alphach)}(\forall n\geq 0,\,X_n\cdot\vecu\geq 0)\Big).\] 
However, by central symmetry, 
\[P_\much^{(\alphach)}(\forall n\geq 0,\,X_n\cdot\vecu\geq 0)=P^{(\alpha)}_\much(\forall n\geq 0,\,X_n\cdot\vecu\leq 0)\]
and the latter event has probability 0 because of Theorem~\ref{thm:main} combined to a 0-1 law (like in corollary, \cite{Zerner-Merkl} in dimension 2, or \cite{Bouchet} in dimension $\ge3$). This concludes.

\subsection{Quenched identity on a cylinder}

Let us first recall the quenched version of Lemma~\ref{lem:cycle}, for which we refer to~\cite{Sabot09} or~\cite{SabotTournier10}. 

\begin{lemme}\label{lem:cycle2}
Let $G=(V,E)$ be a finite directed graph, endowed with positive weights $(\alpha_e)_{e\in E}$. Recall notations from Lemma~\ref{lem:cycle}. To any environment $\omega$ on $G$, we also associate its reverse $\omegach$ defined by $\omegach_\ech=\frac{\pi(\underline e)}{\pi(\overline e)}\omega_e$ for all $e\in E$, where $\pi$ is the invariant measure for $\omega$.  
Assume that $\mathrm{div}(\alpha)=0$. Then $\mathcal L_{\bP^{(\alpha)}}(\omegach)=\bP^{(\alphach)}$. 
\end{lemme}

Let us follow the same steps as for the annealed identity, now with fixed $N=1$, which we omit from indices. 

Let $L\in\N^*$. We have, for any environment $\omega$ on $G_{L}$,
\begin{align*} 
P^\omega_{\omega(\partial,\cdot)}(H_R<H_\partial)
	&= P^\omega_\partial(H_R<H_\partial)= 1- P^\omega_\partial(H^+_\partial<H_R).
\end{align*}
%(note that the two notations $\omega(\partial,\cdot)$ are compatible in a convenient abuse of notation). 
Furthermore, the latter event involves cycles and, considering the reversed cycles, we have
\[P^\omega_\partial(H^+_\partial<H_R) = P^\omegach_\partial(H^+_\partial<H_R),\]
as a consequence of two facts: first, the set of cycles in the left event is globally unchanged after reversal, and second the probability of a cycle in $\omega$ is equal to the probability of its reverse in $\omegach$, as a consequence of the definition of $\omegach$. 
%\[\mathcal{L}_{\bP^{(\alpha^{N,L})}}\Big(P^\omega_\partial(H^+_\partial<H_R)\Big) = \mathcal{L}_{\bP^{(\alphach^{N,L})}}\Big(P^\omegach_\partial(H^+_\partial<H_R)\Big).\] 

For any environment $\omega$ on $G_{L}$, $\omegach$ is an environment on $\Gch_{L}$ hence we may decompose as before, applying Markov property at time 1, 
\[P^\omegach_\partial(H^+_\partial<H_R)=(1-\omegach_{(\partial,R)})P^\omegach_{\omegach(\partial,\cdot)_{|\mathcal L}}(H_\partial<H_R)\]
where $\omegach(\partial,\cdot)_{|\mathcal L}$ is the law of $X_1$ under $P_\partial^\omegach$ conditioned on $\{X_1\neq R\}$. By Lemma~\ref{lem:cycle2}, under $\bP^{(\alpha^{L})}$, $\omegach\sim\bP^{(\alphach^{L})}$. As a consequence, and because of the ``restriction property'' of Dirichlet distribution (cf.~for instance~\cite{Tournier09}), under $\bP^{(\alpha^{L})}$, $1-\omegach_{(\partial,R)}$ is independent of $\omegach(\partial,\cdot)_{|\mathcal L}$ and the latter follows a Dirichlet distribution with parameters $\alphach_{(\partial,\cdot)}$. %, i.e.~it has same law as $\omega_{(\partial,\cdot)}$ under $\bP^{(\alphach^{N,L})}$. 
On the other hand, under $\bP^{(\alpha^{L})}$, 
\begin{align*}
1-\omegach_{(\partial,R)}\sim{\rm Beta}\left(\alphach_{(\partial,\mathcal L)},\alphach_{(\partial,R)}\right)
	& ={\rm Beta}\left(\alpha_{(\mathcal L,\partial)},\alpha_{(R,\partial)}\right),
	%& ={\rm Beta}\left(M\Sigma\cdot E_o^\omega[(X_1\cdot\vecu)_-],M\Sigma\cdot E_o^\omega[X_1\cdot\vecu]\right).
\end{align*}
which is the distribution in~\eqref{eq:identity_cylinder}.

Gathering everything, we obtain that the law under $\bP^{(\alpha^{L})}$ of $P^\omega_{\omega(\partial,\cdot)}(H_R<H_\partial)$ is the same as the law of
\begin{equation}\label{eqn:proof_cylinder}
1-(1-\omega_{(R,\partial)})\big(1-P^\omega_{\omega(\partial,\cdot)}(H_R<H_\partial)\big)
\end{equation}
under $\bP^{(\alphach^{L})}$ (note that here~$\omega$ is an environment on $\Gch$). Although this is not necessary, we may note that the two factors are independent, because the paths involved in the last event don't go out of vertex~$R$. 

As was noticed in the annealed proof, when $L$ goes to infinity, the expectation under $\bP^{(\alphach^{L})}$ of the last probability in~\eqref{eqn:proof_cylinder} goes to $P^{(\alpha)}_\much(\forall n,\,X_n\cdot\vecu\le0)=0$, hence the last probability under $\bP^{(\alphach^{L})}$ goes to $0$ in $L^1$ and thus in law. On the other hand, the law of $\omega_{(R,\partial)}$ under $\bP^{(\alphach^{L})}$ was shown above to be the Beta distribution from~\eqref{eq:identity_cylinder}, and thus does not depend on $L$.

We conclude that the law under $\bP^{(\alpha^{L})}$ of $P^\omega_{\omega(\partial,\cdot)}(H_R<H_\partial)$ converges to the Beta distribution given in~\eqref{eq:identity_cylinder}. This is the expected conclusion since, on the other hand, these quenched probabilities for growing $L$ can be expressed on the same cylinder $C$ and thus seen to converge as $L\to\infty$:
\[P^\omega_{\omega(\partial,\cdot)}(H_R<H_\partial)=P^\omega_{\omega(\partial,\cdot)}(T^\vecu_L<\widetilde{T}^\vecu_0)\limites{}{L\to\infty} P^\omega_{\omega(\partial,\cdot)}(\forall n, X_n\cdot\vecu\ge0\text{, and }\limsup_n X_n\cdot\vecu=+\infty).\] 
As before, the event $\{\limsup_n X_n\cdot\vecu=+\infty\}$ is $\bP^{(\alpha)}$-a.s.~included in $\{\forall n,X_n\cdot\vecu\ge0\}$ because of Lemma~4 of~\cite{Zerner-Merkl}.


\begin{thebibliography}{2}
%   \bibitem[BoZe07]{Bolthausen-Zeitouni} \textsc{Bolthausen, E. and Zeitouni, O.} (2007) Multiscale analysis of exit distributions for random walks in random environments. \emph{Probab. Theory Related Fields} 138, no. 3--4, 581--645. \htmladdnormallink{MR1978990}{http://www.ams.org/mathscinet-getitem?mr=MR1978990}
   \bibitem[Bo12]{Bouchet} \textsc{Bouchet, E.} (2012) Sub-ballistic random walk in Dirichlet environment. \emph{Preprint. \htmladdnormallink{arXiv:1205.5709}{http://arxiv.org/abs/1205.5709}}
	\bibitem[ChLe91]{Chamayou-Letac} \textsc{Chamayou, J.-F. and Letac G.} (1991) Explicit stationary distributions for compositions of random functions and products of random matrices. \emph{J. Theoret. Probab.} 4, 3--36. \htmladdnormallink{MR1088391}{http://www.ams.org/mathscinet-getitem?mr=1088391}
   \bibitem[DrRa10]{DrewitzRamirez} \textsc{Drewitz, A. and Ram\'irez, A.} (2010) Asymptotic direction in random walks in random environment revisited.  \emph{Braz. J. Probab. Stat.} 24, no. 2, 212--225.  \htmladdnormallink{MR2643564}{http://www.ams.org/mathscinet-getitem?mr=MR2643564}
   \bibitem[EnSa02]{EnriquezSabot02} \textsc{Enriquez, N. and Sabot, C.} (2002) Edge oriented reinforced random walks and RWRE. \emph{C. R. Math. Acad. Sci. Paris} 335, no. 11, 941--946. \htmladdnormallink{MR1952554}{http://www.ams.org/mathscinet-getitem?mr=MR1952554}
   \bibitem[EnSa06]{EnriquezSabot06} \textsc{Enriquez, N. and Sabot, C.} (2006) Random walks in a Dirichlet environment. \emph{Electron. J. Probab.} 11, no. 31, 802--817 (electronic). \htmladdnormallink{MR2242664}{http://www.ams.org/mathscinet-getitem?mr=MR2242664}. 
%   \bibitem[Ka81]{Kalikow} \textsc{Kalikow, S.} (1981) Generalized random walk in a random environment. \emph{Ann. Probab.} 9, no. 5, 753--768. \htmladdnormallink{MR0628871}{http://www.ams.org/mathscinet-getitem?mr=MR0628871} 
   \bibitem[ZeMe01]{Zerner-Merkl} \textsc{Zerner, M. and Merkl, F.} (2001) A zero-one law for planar random walks in random environment. \emph{Ann. Probab.} 29, no. 4, 1716--1732. \htmladdnormallink{MR1880239}{http://www.ams.org/mathscinet-getitem?mr=MR1880239}
   \bibitem[Pe88]{Pemantle} \textsc{Pemantle, R.} (1988) Phase transition in reinforced random walk and RWRE on trees. \emph{Ann. Probab.} 16, no. 3, 1229--1241. \htmladdnormallink{MR0942765}{http://www.ams.org/mathscinet-getitem?mr=MR0942765}
%   \bibitem[Pe88]{Pemantle} \textsc{Pemantle, R.} (1988) Random processes with reinforcement. PhD thesis, MIT. (available online at \url{http://www.math.upenn.edu/~pemantle/papers/thesis.html})
%   \bibitem[Pe07]{Pemantle07} \textsc{Pemantle, R.} (2007) A survey of random processes with reinforcement. \emph{Probab. Surv.} 4, 1--79 (electronic). \htmladdnormallink{MR2282181}{http://www.ams.org/mathscinet-getitem?mr=MR2282181}
   \bibitem[Sa04]{Sabot04} \textsc{Sabot, C.} (2004) Ballistic random walks in random environments at low disorder. \emph{Ann. Probab.} 32, no. 4, 2996--3023. \htmladdnormallink{MR2094437}{http://www.ams.org/mathscinet-getitem?mr=MR2094437}
   \bibitem[Sa09]{Sabot09} \textsc{Sabot, C.} (2009) Random walks in random Dirichlet environment are transient in dimension $d\geq 3$. \emph{Probab.\ Theory Relat.\ Fields} 151, nos.\ 1--2, 297--317. \htmladdnormallink{MR2834720}{http://www.ams.org/mathscinet-getitem?mr=MR2834720}
    \bibitem[Sa10]{Sabot10} \textsc{Sabot, C.} (2010) Random Dirichlet environment viewed from the particle in dimension $d\ge 3$. \emph{To be published in Annals of Probability}
    \bibitem[SaTo10]{SabotTournier10} \textsc{Sabot, C.\ and Tournier, L.} (2011) Reversed Dirichlet environment and directional transience of random walks in Dirichlet environment. \emph{Ann.~Inst.~Henri Poincar\'e Probab.~Stat.} 47, no.\ 1, 1--8.  \htmladdnormallink{MR2779393}{http://www.ams.org/mathscinet-getitem?mr=MR2779393}
    \bibitem[Si07]{Simenhaus} \textsc{Simenhaus, F.} (2007) Asymptotic direction for random walks in random environment. \emph{Ann.~Inst.\ Henri Poincar\'e Probab.\ Stat.} 43, no.\ 6, 751--761. 
%   \bibitem[So75]{Solomon75} \textsc{Solomon, F.} (1975) Random walks in a random environment. \emph{Ann. Probab.} 3, 1--31. 
%   \bibitem[SzZe99]{SznitmanZerner} \textsc{Sznitman, A.-S. and Zerner, M.} (1999) A law of large numbers for random walks in random environment. \emph{Ann. Probab.} 27, no. 4, 1851--1869. \htmladdnormallink{MR1742891}{http://www.ams.org/mathscinet-getitem?mr=MR1742891}
%   \bibitem[Sz01]{Sznitman01} \textsc{Sznitman, A.-S.} (2001) On a class of transient random walks in random environment. \emph{Ann. Probab.} 29, no. 2, 724--765. \htmladdnormallink{MR1849176}{http://www.ams.org/mathscinet-getitem?mr=MR1849176}
%   \bibitem[Sz04]{Sznitman02} \textsc{Sznitman, A.-S.} (2002) An effective criterion for ballistic behavior of random walks in random environment. \emph{Probab. Theory Related Fields} 122, 509--544. \htmladdnormallink{MR1902189}{http://www.ams.org/mathscinet-getitem?mr=MR1902189}
%   \bibitem[Sz04]{Sznitman} \textsc{Sznitman, A.S.} (2004) Topics in random walk in random environment. Notes of course at School and Conference on Probability Theory, May 2002, \emph{ICTP Lecture Series}, Trieste, pp. 203--266. (available online at \url{users.ictp.trieste.it/~pub_off/lectures/lns017/Sznitman/Sznitman.ps.gz}) 
   \bibitem[To09]{Tournier09} \textsc{Tournier, L.} (2009) Integrability of exit times and ballisticity for random walks in Dirichlet environment. \emph{Electron. J. Probab.} 14, no. 16, 431--451 (electronic). \htmladdnormallink{MR2480548}{http://www.ams.org/mathscinet-getitem?mr=MR2480548}
   %\bibitem[Ze01]{Zeitouni} \textsc{Zeitouni, O.} (2001) Lecture notes on random walks in random environment. \'Ecole d'\'et\'e de probabilit\'es de Saint-Flour. (available online at \url{http://www.ee.technion.ac.il/~zeitouni/})
   \bibitem[Ze07]{Zerner} \textsc{Zerner, M.} (2007) The zero-one law for planar random walks in i.i.d. random environments revisited. \emph{Electron. Comm. Probab.}  12, 326--335 (electronic). \htmladdnormallink{MR2342711}{http://www.ams.org/mathscinet-getitem?mr=MR2342711} 
\end{thebibliography}
\end{document}